\documentclass[12pt]{amsart}

\usepackage{a4wide}
\usepackage{amsthm,amsfonts,amsmath,mathrsfs,amssymb}
\usepackage{dsfont}
\usepackage[T1]{fontenc}
\usepackage[utf8]{inputenc}
\usepackage{pdfpages}
\usepackage{graphicx}

\def\a{\alpha}               
           
\def\D{{\mathbb D}}  
\def\C{{\mathbb C}}

\def\({\left(}       \def\){\right)}

\newtheorem{prop}{\sc Proposition}

\newtheorem{thm}{\sc Theorem}

\begin{document}
\title[On the Schwarzian derivative]
{Harmonic Schwarzian derivative and methods of approximation of zeros}
\dedicatory{Dedicated to Professors Antonio Martin\'on and Jos\'e M. M\'endez on the occasion of their retirement}
\author[M. J. Mart\'{\i}n]{Mar\'{\i}a J. Mart\'{\i}n}
\address{Departamento de An\'alisis Matem\'atico, Universidad de La Laguna.  Av. Astrof\'{\i}sico Francisco S\'anchez, s/n. 38271, La Laguna (Sta. Cruz de Tenerife), Spain.}\email{maria.martin@ull.es}


\begin{abstract} We review the relation between the classical formulas of the pre-Schwarzian and Schwarzian derivatives of locally univalent analytic functions and the derivatives of the generating functions of the methods due to Newton and Halley, respectively, for approximating zeros. We extend these relations to the cases when the functions considered are harmonic.
\end{abstract}

\maketitle
\date{\today}
\section*{Introduction}

Let $\varphi$ be a holomorphic function in the simply connected domain $\Omega\subset \mathbb{C}$. It is well known that such function is \emph{locally univalent} (\emph{i.e.} locally injective)  if and only if $\varphi'$ is different from zero in $\Omega$. Equivalently, if the Jacobian of $\varphi$, given by $J_\varphi=|\varphi'|^2$, does not vanish in $\Omega$. 

\par\smallskip
For such a  holomorphic function  $\varphi$, the \emph{pre-Schwarzian derivative}, $P(\varphi)$, of $\varphi$ equals
\begin{equation}\label{def-preS}
P(\varphi)=\frac{\varphi''}{\varphi'}\,.
\end{equation}
Closely related is the so-called \emph{Schwarzian derivative}, $S(\varphi)$, of $\varphi$, given by
\begin{equation}\label{def-S}
S(\varphi)=(P(\varphi))' -\frac 12 (P(\varphi))^2=\left(\frac{\varphi''}{\varphi'}\right)'-\frac 12 \left(\frac{\varphi''}{\varphi'}\right)^2=\frac{\varphi'''}{\varphi'}-\frac 32 \left(\frac{\varphi''}{\varphi'}\right)^2\,.
\end{equation}

\par\medskip
One of the key properties of this non-linear operator is its invariance under post-composition with M\"obius transformations. That is, $S(M\circ\varphi)=S(\varphi)$ for all \emph{M\"obius} (or linear fractional) transformations
\[
M(z)=\frac{az+b}{cz+d}\,,\quad ad-bc\neq 0\,.
\]
\par\smallskip
Based on this characteristic, and according to \cite{D-harm}, here is the derivation of \eqref{def-S} due to Hermann Amandus Schwarz in $1873$.
\par\smallskip
Let $\varphi$ and $\psi$ be two locally univalent holomorphic functions such that $\psi=M\circ \varphi$ for some M\"obius transformation $M$ as above, so that $(c\varphi+d)\psi=a\varphi+b$. Three successive differentiations produce the system of linear equations
\[
\begin{cases}
c(\varphi\psi)'+d\psi'-a\varphi' & =  0\,,\\
c(\varphi\psi)''+d\psi''-a\varphi'' & =  0\,,\\
c(\varphi\psi)'''+d\psi'''-a\varphi''' & =  0\,,\\
\end{cases}
\]
with nontrivial solution $(c, d, -a)$. The existence of a nontrivial solution guarantees that the determinant of the coefficients of the linear system  vanishes identically. When the determinant is expanded, the equation reads as
\[
3\psi'^2\varphi''^2+2\psi'\psi'''\varphi'^2=3\varphi'^2\psi''^2+2\varphi'\varphi'''\psi'^2\,.
\]
Dividing both sides  by $2\varphi'^2\psi'^2$, we obtain 
\[
\frac{\varphi'''}{\varphi'}-\frac 32 \left(\frac{\varphi''}{\varphi'}\right)^2=\frac{\psi'''}{\psi'}-\frac 32 \left(\frac{\psi''}{\psi'}\right)^2\,,
\]
which says that both functions $\varphi$ and $\psi$ must satisfy the equation $S(\varphi)=S(\psi)$, where $S(\varphi)$ and $S(\psi)$ are, respectively, the Schwarzian derivatives of $\varphi$ and $\psi$ given by \eqref{def-S}.

\par\medskip
The Schwarzian derivative of a locally univalent function $\varphi$ has been used historically to measure ``how close'' is $\varphi$ to being a M\"obius transformation or (in an invariant form) how  close it is to being univalent. We refer the reader to the beautiful paper \cite{Osgood} as a detailed source of historical information about the Schwarzian derivative and its different properties and applications. Another interesting paper is \cite{OT}, where the authors mention that the formula for the Schwarzian derivative was discovered by Lagrange in his treatise ``\emph{Sur la construction des cartes g\`eographiques}'' from $1781$. It also appears in the paper \cite{Kummer}, published in 1836. The denomination of the operator defined by \eqref{def-S} as ``Schwarzian derivative'' (or simply, ``Schwarzian'') after Schwarz is attributed to Cayley. 
\par\smallskip
But, in fact, it turns out that the formula \eqref{def-S} is related with the error term in the iterative method due to Halley (presented in his 1694 paper \cite{Halley}) for approximating zeros of a given function.  Before explaining this relation, the reader should be advised of the fact that, in this work, we are not analyzing the conditions of convergence of the methods we present but just the relations between the formulas \eqref{def-preS} and \eqref{def-S} and the derivatives of the so-called generating functions of the methods.

\subsection*{The Newton and Halley Methods for approximating the zeros of a function} It is well know that the Newton method (properly known as the Newton-Raphson method) is a technique to find an approximation for the root of a  function $f$. Roughly speaking, the technique starts with a guess $z_0$, say, and uses an iterative procedure to find a sequence of improved guesses. In particular, if $z_n$ is “close” to a root of $\varphi$, then
\[
z_{n+1}=z_n-\frac{\varphi(z_n)}{\varphi'(z_n)}\,,
\]
will generally be (much) closer. 

The \emph{generating function} $F^N$ of the method is then given by
\[
F^N(z)=z-\frac{\varphi(z)}{\varphi'(z)}\,.
\]
A straightforward calculation shows that if $\alpha$ is one zero of the locally univalent analytic function $\varphi$, then $F^N(\alpha)=\alpha$, $(F^N)'(\alpha)=0$ and $(F^N)''(\alpha)=P(\varphi)(\alpha)$, where $P(\varphi)$ is given by \eqref{def-preS}.

\par\medskip
Notice that the formula for the function $F^N$  can be alternatively obtained by applying the following procedure.  

Let $\varphi$ have a zero at the point $\alpha$. Then, by using the first two terms in the Taylor  series expansion of $\varphi$ centered at a point $z$ (close to $\alpha$), we obtain
\[
0=\varphi(\alpha)\approx  \varphi(z)+\varphi'(z)(\alpha-z)\,,
\]
which gives
\begin{equation}\label{eq-approxN}
\alpha -z \approx  -\frac{\varphi(z)}{\varphi'(z)}\,, 
\end{equation}
or, in other words, we have
\[
\alpha\approx F^N(z)\,.
\]

The recursive formula in Halley's method to approximate simple zeros of  functions $\varphi$ is

\[
z_{n+1}=z_n-\frac{2\varphi(z_n)\varphi'(z_n)}{2[\varphi'(z_n)]^2-\varphi(z_n)\varphi''(z_n)}\,,
\]
beginning with an initial guess $z_0$. In this case, the corresponding generating function of the method equals
\[
F^H(z)=z-\frac{2\varphi(z)\varphi'(z)}{2[\varphi'(z)]^2-\varphi(z)\varphi''(z)}\,.
\]

This is precisely the function obtained by arguing in a similar way as before by involving the Taylor series expansion of $\varphi$ (which is, in fact, a similar approach to the one presented in \cite{Palmore}; see also \cite{Brown}). Namely, let $\varphi$ have a zero at the point $\alpha$. By using the first three terms in the Taylor series expansion of $\varphi$ centered at a point $z$ close to $\alpha$, we have (using also the approximation \eqref{eq-approxN} from the Newton method)
\begin{eqnarray*}
0&=&\varphi(\alpha) \approx  \varphi(z)+\varphi'(z)(\alpha-z)+\frac{\varphi''(z)}{2}(\alpha-z)^2\\
&=&  \varphi(z)+(\alpha-z)\left[\varphi'(z)+\frac{\varphi''(z)}{2}(\alpha-z)\right]\\
&\approx&  \varphi(z)+(\alpha-z)\left[\varphi'(z)-\frac{\varphi''(z)\varphi(z)}{2\varphi'(z)}\right]\,,
\end{eqnarray*}
which gives
\begin{eqnarray*}
\alpha &\approx& z-\frac{\varphi(z)}{\varphi'(z)-\frac{\varphi(z)\varphi''(z)}{2\varphi'(z)}}\\
&=& z-\frac{2\varphi(z)\varphi'(z)}{2[\varphi'(z)]^2-\varphi(z)\varphi''(z)}=F^H(z)\,.
\end{eqnarray*}

It is clear that $F^H(\alpha)=\alpha$. The first two derivatives of this function at  $\alpha$ equal
\[
(F^H)'(\alpha)=(F^H)''(\alpha)=0\,.
\]

The Schwarzian derivative of $\varphi$ comes up when calculating the third derivative of $F^H$. Concretely, it is not difficult to check that
\[
 (F^H)'''(\alpha)=-S(\varphi)(\alpha)\,,
\]
where, as announced, $S(\varphi)$ is the Schwarzian derivative of $\varphi$ given by \eqref{def-S}.

\par\medskip

The author (jointly with R. Hern\'andez)  introduced  in \cite{HM-Schwarzian} the definition of the \emph{harmonic pre-Schwarzian} and \emph{Schwarzian derivatives} of locally univalent harmonic mappings. Different approaches where analyzed to get the definition of these operators, as is explained in \cite{HM-Schwarzian}.  On the other hand, the authors in \cite{SZ} have recently presented a ``harmonic'' Newton  method for approximation of zeros for harmonic mappings. In this paper, we show that the generating function of method presented in \cite{SZ} can be obtained by involving the series expansion of the harmonic mapping in the same way as above (for analytic functions). Based on this idea, we present a harmonic analogue of the generating function for what could be a ``harmonic Halley method''. We prove that, in fact, the harmonic pre-Schwarzian and Schwarzian derivatives of the harmonic function $f$ are related to the derivatives of the generating functions of the harmonic methods at the zero of $f$.   

In other words, we generalize the relation between the classical pre-Schwarzian and Schwarzian derivatives of analytic functions and the generating functions of the methods for approximating zeros due to Newton and Halley to those cases when the functions considered are merely harmonic. 

\section{The harmonic pre-Schwarzian and Schwarzian derivatives}

A complex-valued function $f=u+iv$ is \emph{harmonic} in a domain $\Omega\subset \C$ if both the real and imaginary parts of $f$ are (real) harmonic in $\Omega$. That is, if and only if $\Delta u = \Delta v=0$, where $\Delta$ is the \emph{Laplacian}, defined by
\[
\Delta=\frac{\partial^2}{\partial x^2}+ \frac{\partial^2}{\partial y^2}\,.
\]

In other words, harmonic mappings are complex-valued harmonic functions whose real and imaginary parts are not necessarily conjugate. Therefore, the Cauchy-Riemann equations need not be satisfied, so the functions need not be analytic. 
\par\medskip
As is mentioned in \cite{D-harm}, although harmonic mappings are natural generalizations of analytic functions, they were studied originally by differential geometers because of their natural role in parametrizing minimal surfaces. The seminal paper \cite{C-SS}, where the authors point out that many of the classical results for conformal mappings have clear analogues for univalent harmonic mappings, made the complex analysts become attracted by this type of functions.  

\par\smallskip
The \emph{Wirtinger operators}, commonly used in complex analysis, are very convenient. They are defined by
\begin{equation}\label{eq-Wirtinger}
\frac{\partial}{\partial  z}=\frac{1}{2}\left(\frac{\partial}{\partial x}-i \frac{\partial}{\partial  y}\right) \quad \text{and}\quad \frac{\partial}{\partial \overline z}=\frac{1}{2}\left(\frac{\partial}{\partial  x}+i \frac{\partial}{\partial  y}\right)\,,
\end{equation}
where $z=x+iy$. A direct calculation shows that, in terms of these complex derivatives, the Laplacian of $f$ is 
\[
\Delta f =4 \frac{\partial^2 f}{\partial z\partial \overline z}\,.
\]

This latter expression for the Laplacian allows to prove that every harmonic mapping in a simply connected domain $\Omega$ has a \emph{canonical decomposition} $f=h+\overline g$, where $h$ and $g$ are analytic functions in $\Omega$. This decomposition is unique up to an additive constant (see \cite[p. 7]{D-harm}).

\par\smallskip
The fact that a harmonic function $f=h+\overline g$ is locally univalent if and only if its Jacobian $J_f=|h'|^2-|g'|^2$ does not vanish in $\Omega$  is due to Lewy \cite{Lewy}. A locally univalent harmonic mapping is said to be \emph{orientation-preserving} if $|h'|>|g'|$ in $\Omega$ or, equivalently, if $h$ is a locally univalent analytic function and the (second complex) \emph{dilatation} $\omega=g'/h'$ is analytic and satisfies the condition $\omega(\Omega)\subset\D$, where $\D$ is the unit disk in the complex plane. Otherwise, the locally univalent function is \emph{orientation-reversing}. It is obvious that $f$ is orientation-reversing if and only if $\overline f$ is orientation-preserving.

\par\smallskip
The \emph{harmonic Schwarzian derivative} $S_H$ of a locally univalent harmonic function $f$ with Jacobian $J$ was defined in \cite{HM-Schwarzian} as
\begin{equation}\label{eq-Schwarzian0}
S_H(f)=\frac{\partial}{\partial z}\left(P_H(f)\right)-\frac 12 \left(P_H(f)\right)^2\,,
\end{equation}
where $P_H(f)$ is the \emph{harmonic pre-Schwarzian derivative} of $f$, which equals
\[
P_H(f)=\frac{\partial}{\partial z} \log J\,.
\]
It is easy to check that $P_H(f)=P_H(\overline f)$ and $S_H(f)=S_H(\overline f)$ for any locally univalent harmonic mapping $f$ in a simply connected domain $\Omega$. Therefore, without loss of generality we may assume that $f$ is orientation-preserving.  

It is convenient to use the operators in \eqref{eq-Wirtinger} and notice that the following identities hold for a given analytic function $h$ (or $g$)
\begin{equation}\label{eq-derivadas}
\frac{\partial h}{\partial z} = h'\quad \text{and}\quad \frac{\partial \overline h}{\partial z}=0
\end{equation}
to get that the harmonic pre-Schwarzian and Schwarzian derivatives of the orientation-preserving harmonic mapping $f=h+\overline g$ with dilatation $\omega=g^\prime/h^\prime$ can be written, respectively,  as
\begin{equation}\label{eq-preSchwarzian}
P_H(f)=P(h)-\frac{\overline \omega \omega'}{1-|\omega|^2} = \frac{h''}{h'}-\frac{\overline \omega \omega'}{1-|\omega|^2}\,,
\end{equation}
and
\begin{equation}\label{eq-Schwarzian}
S_H(f)=S(h)+\frac{\overline \omega}{1-|\omega|^2}\left(\frac{h''}{h'}\,\omega'-\omega''\right)
-\frac 32\left(\frac{\omega'\,\overline
\omega}{1-|\omega|^2}\right)^2\,,
\end{equation}
where $S(h)$ is the classical Schwarzian derivative of the function $h$ defined by \eqref{def-S}.
\par\smallskip
It is clear that if $f$ is analytic (so that its dilatation is identically zero) then $S_H(f)$ coincides with the classical Schwarzian derivative of $f$. In other words, the operator defined by \eqref{eq-Schwarzian0} -or by \eqref{eq-Schwarzian}- can be considered as a generalization of the classical Schwarzian derivative~\eqref{def-S}.  It is also obvious that the harmonic pre-Schwarzian derivative generalizes \eqref{def-preS}. We refer the reader to \cite{HM-Schwarzian} for the motivation of the definition  as well as for different properties that the harmonic Schwarzian (and the harmonic pre-Schwarzian) derivatives satisfy. 

\par\smallskip
These two harmonic operators $P_H$ and $S_H$ have proved to be useful for generalizing classical results related to holomorphic functions to the more general setting of harmonic mappings (see, for instance, \cite{HM-Chuaqui,  HM-qc, HM-Schwarzian, HM-Nehari,   HM-Mobius, Huusko-M}).

\section{The ``harmonic Newton iteration''}

By exploiting the fact that for a given  continuously Fr\'echet differentiable map $F:D\subset X\to Y$ between two Banach spaces $X$ and $Y$ (where $D$ is an open set in $X$) the \emph{Newton iteration} with initial point $z_0\in D$ is given by 
\[
z_{n+1}=z_n-[F'(z_n)]^{-1}F(z_n)\,,\quad n\geq 0\,,
\]
the authors in \cite{SZ} define the \emph{harmonic Newton iteration} for a harmonic mapping $f=h+\overline g$ as
\begin{eqnarray*}
z_{n+1}&=& z_n-[f'(z_n)]^{-1}f(z_n)\\
&=& z_n-\frac{\overline{h'(z_n)}f(z_n)-\overline{g'(z_n)f(z_n)}}{J_f(z_n)}=z_n-\frac{\overline{h'(z_n)}f(z_n)-\overline{g'(z_n)f(z_n)}}{|h'(z_n)|^2-|g'(z_n)|^2}\,,
\end{eqnarray*}
beginning with an initial guess $z_0$. Therefore, the generating function of the method (or the \emph{harmonic Newton map}, as named in \cite{SZ}) equals
\begin{equation}\label{eq-Newtongenerating}
F^{N_H}(z)=z-\frac{\overline{h'(z)}f(z)-\overline{g'(z)f(z)}}{|h'(z)|^2-|g'(z)|^2}\,.
\end{equation}

Let us again argue as in the introduction to show an alternative way of getting $F^{N_H}$ by involving the power series expansion of $f$. For a given zero $\alpha$ of $f=h+\overline g$, let $z$ be close enough to $\alpha$ so that, by using only the first terms in the power series of $f$, we can write
\begin{equation}\label{eq-N11}
0=f(\alpha)\approx f(z)+h'(z)(\alpha-z)+\overline{g'(z)}(\overline{\alpha-z})\,.
\end{equation}
By taking complex conjugates, we obtain
\begin{equation}\label{eq-N12}
0\approx \overline{f(z)}+\overline{h'(z)}(\overline{\alpha-z})+g'(z)(\alpha-z)\,.
\end{equation}

After multiplying \eqref{eq-N11} by $\overline{h'(z)}$, \eqref{eq-N12} by $\overline{g'(z)}$, and subtracting the resulting equations, we have
\[
0\approx \overline{h'(z)} f(z)-\overline{g'(z)f(z)}+\left(|h'(z)|^2-|g'(z)|^2\right) (\alpha-z)\,.
\]
That is,
\begin{equation}\label{eq-aN}
\alpha-z \approx -\frac{\overline{h'(z)} f(z)-\overline{g'(z)f(z)}}{|h'(z)|^2-|g'(z)|^2}
\end{equation}
or, in other words, 
\[
\alpha \approx z -\frac{\overline{h'(z)} f(z)-\overline{g'(z)f(z)}}{|h'(z)|^2-|g'(z)|^2}=F^{N_H}(z)\,,
\]
where $F^{N_H}$ is the function in \eqref{eq-Newtongenerating}.

The following proposition shows the relation between the derivatives of the generating function of the harmonic Newton method and the harmonic pre-Schwarzian derivative $P_H$ defined in \eqref{eq-preSchwarzian} at the zero of the function $f$. The result generalizes this relation for analytic functions (mentioned in the introduction) to harmonic mappings. 

\begin{prop}
Let $f=h+\overline g$ be a (locally univalent) harmonic function in a disk centered at $\alpha$ with $f(\alpha)=0$ and let $F^{N_H}$ be the generating function of the harmonic Newton method given by \eqref{eq-Newtongenerating}. Then,
\[
F^{N_H}(\alpha)=\alpha\,,\quad \frac{\partial F^{N_H}}{\partial z}(\a)=F^{N_H}_z(\alpha)=0\,,\quad\text{and}\quad \frac{\partial^2 F^{N_H}}{\partial z^2}(\a)=F^{N_H}_{zz}(\alpha)=P_H(f)(\alpha)\,,
\]
where $P_H$ is the harmonic pre-Schwarzian derivative defined in \eqref{eq-preSchwarzian}
\end{prop}

\begin{proof}
The fact that $f$ is locally univalent (so that the Jacobian $J=|h'|^2-|g'|^2$ of $f$ is different from $0$) guarantees that $F^{N_H}$ is well defined in a disk centered at $\alpha$ and with positive radius. In fact, $F^{N_H}$ is a function of class $\mathcal C^\infty$ in the variables $z$ and $\overline z$. 

It is obvious that $F^{N_H}(\alpha)=\alpha$ if $f(\alpha)=0$. To compute the derivatives of $F^{N_H}$ with respect to $z$, it is convenient to use the operators in \eqref{eq-Wirtinger} and \eqref{eq-derivadas}. We get

\begin{eqnarray*}
F^{N_H}_z(z)&=&1-\frac{\left(\overline{h'(z)}h'(z)-\overline{g'(z)}g'(z)\right)J(z)-\left(\overline{h'(z)} f(z)-\overline{g'(z)f(z)}\right)J_z(z)}{J^2(z)}\\
&=&  \frac{\left(\overline{h'(z)} f(z)-\overline{g'(z)f(z)}\right)J_z(z)}{J^2(z)}\,.
\end{eqnarray*}
Therefore, we have $F^{N_H}_z(\alpha)=0$, since $f(\alpha)=0$.
\par\smallskip
A straightforward calculation gives
\begin{eqnarray*}
F^{N_H}_{zz}(\alpha)&=&\frac{J_z(\alpha)}{J(\alpha)}=\frac{\overline{h'(\alpha)}h''(\alpha)-\overline{g'(\alpha)}g''(\alpha)}{\overline{h'(\alpha)}h'(\alpha)-\overline{g'(\alpha)}g'(\alpha)} \,.
\end{eqnarray*}

To finish the proof, it suffices to recall that $g'=\omega h'$, where $\omega$ is the dilatation of the function $f=h+\overline g$. This identity implies that  $g''=\omega'h'+\omega h''$ and, therefore, we can write 
\begin{eqnarray*}
F^{N_H}_{zz}(\alpha)&=&\frac{\overline{h'(\alpha)}h''(\alpha)-\overline{\omega(\alpha)h'(\alpha)}\left(\omega'(\a)h'(\a)+\omega(\a) h''(\a)\right)}{\overline{h'(\alpha)}h'(\alpha)-\overline{\omega(\alpha)h'(\alpha)}\omega(\alpha)h'(\alpha)}\\
&=& \frac{h''(\alpha)\left(1-|\omega(\alpha)|^2\right)-\overline{\omega(\alpha)}\omega'(\a)h'(\a)}{h'(\alpha)\left(1-|\omega(\alpha)|^2\right)}\\
&=&\frac{h''(\alpha)}{h'(\alpha)}-\frac{\overline{\omega(\alpha)}\omega'(\alpha)}{1-|\omega(\alpha)|^2}=P_H(f)(\alpha)\,.
\end{eqnarray*}
\end{proof}

\section{The ``harmonic Halley iteration''}

We have not been able to find an analogue of the classical Halley method for approximation of zeros or harmonic functions in the literature. However, motivated by the fact that the use of the truncated power series expansion of the function gives rise to the so-called harmonic Newton map introduced in the paper \cite{SZ} (as explained in the previous section) and that the same approach also works in the classical Newton and Halley methods (as explained in the introduction), we propose the following formula for what could be a generating function for a ``harmonic Halley method''. 

\par\smallskip
Let $f=h+\overline g$ be a locally univalent harmonic mapping defined in a simply connected domain $\Omega$ and let $\alpha\in\Omega$ be a zero of $f$. Since $f$ is a harmonic mapping, so that $f_{z\overline z}=0$, for $z$ sufficiently close to $\alpha$, we can write
\begin{eqnarray*}
0&=&f(\alpha)\approx f(z) + h'(z)(\alpha-z) +\overline{g'(z)}(\overline{\a-z})+\frac{h''(z)}{2}(\a-z)^2+ \frac{\overline{g''(z)}}{2}(\overline{\a-z})^2\\
&=& f(z)+(\alpha-z)\left(h'(z)+\frac{h''(z)}{2}(\a-z)\right)+ (\overline{\a-z})\left(\overline{g'(z)}+\frac{\overline{g''(z)}}{2}(\overline{\a-z})  \right)\,.
\end{eqnarray*}

Let us now use \eqref{eq-aN} to get

\begin{eqnarray}\label{eq-H1}
0 &\approx&  f(z)+(\alpha-z)\left(h'(z)-\frac{h''(z)}{2}\frac{N(z)}{J(z)}\right)+ (\overline{\a-z})\left(\overline{g'(z)}-\frac{\overline{g''(z)}}{2}\frac{\overline{N(z)}}{J(z)} \right)\,,
\end{eqnarray}
where $N(z)=\overline{h'(z)} f(z)-\overline{g'(z)f(z)}$ and $J(z)=|h'(z)|^2-|g'(z)|^2$ is the Jacobian of $f$ at $z$.  Taking complex conjugates in the previous equation, we have 

\begin{eqnarray}\label{eq-H2}
0 &\approx&  \overline{f(z)}+(\overline{\alpha-z})\left(\overline{h'(z)}-\frac{\overline{h''(z)}}{2}\frac{\overline{N(z)}}{J(z)}\right)+ (\a-z)\left(g'(z)-\frac{g''(z)}{2}\frac{N(z)}{J(z)} \right)\,.
\end{eqnarray}

By multiplying \eqref{eq-H1} by the coefficient of $(\overline{\a-z})$ in \eqref{eq-H2},  \eqref{eq-H2} by the coefficient of $(\overline{\a-z})$ in \eqref{eq-H1}, and subtracting, we get an expression of the form
\[
0\approx A(z) + B(z) (\alpha-z)\,,
\]
where
\[
A(z)= \left(\overline{h'(z)}-\frac{\overline{h''(z)}}{2}\frac{\overline{N(z)}}{J(z)}\right)f(z)-\left(\overline{g'(z)}-\frac{\overline{g''(z)}}{2}\frac{\overline{N(z)}}{J(z)}\right)\overline{f(z)}
\]
and

\[
B(z)= \left| h'(z)-\frac{h''(z)}{2}\frac{N(z)}{J(z)}\right|^2- \left| g'(z)-\frac{g''(z)}{2}\frac{N(z)}{J(z)}\right|^2\,.
\]

This gives the expression for the generating function of the \emph{harmonic Halley method}:
\begin{equation}\label{eq-fcn}
\alpha\approx z-\frac{A(z)}{B(z)}=F^{H_H}(z)\,.
\end{equation}

We now prove the main result in this paper which shows that $F_{zzz}^{H_H}(\alpha)=-S_H(f)(\alpha)$, where $S_H(f)$ is the harmonic Schwarzian derivative of $f$.

\begin{thm}
For a given locally univalent harmonic mapping $f=h+\overline g$ with $f(\alpha)=0$,  and the function $F_{zzz}^{H_H}$ defined by \eqref{eq-fcn}, the following identities hold: 
\[
F^{H_H}(\a)=\a\,,\quad F^{H_H}_z(\a)=F^{H_H}_{zz}(\a)=0\,,\quad \text{and}\quad F^{H_H}_{zzz}(\a)=-S_H(f)(\alpha)\,,
\]
where $S_H(f)$ is the harmonic Schwarzian derivative of $f$ given by \eqref{eq-Schwarzian}.
\end{thm}

\begin{proof}
In order not to burden the notation, let us agree to use $F$ to denote $F^{H_H}$.  

Let us start by obtaining the following identities related to the function $N=\overline{h'} f-\overline{g'f}$ involved in the definition of the mappings $A$ and $B$ in the formula \eqref{eq-fcn} for $F$, which will be helpful. To do so, we use  \eqref{eq-derivadas}.

\par\smallskip
Bearing in mind that $f(\a)=0$, we have $N(\a)=0$. Hence, $B(\a)=J(\a)$ and, since $A(\a)=0$, we have $F(\a)=\a$.

\par\smallskip

A straightforward calculation shows 
\[
N_z= J\quad \text{and}\quad  (\overline N)_z = h''\overline{f}-g''f\,, 
\]
so that (using again the condition that $f(\a)=0$), we have

\begin{equation} \label{eq-N1}
N_z(\a)=J(\a)\quad \text{and}\quad (\overline N)_z(\a)=0\,.
\end{equation} 

Now, it is easy to check that 

\[
\left(\frac{N}{J}\right)_z=\frac{N_z J-NJ_z}{J^2}=1-\frac{NJ_z}{J^2}\quad \text{and}\quad \left(\frac{\overline N}{J}\right)_z=\frac{(\overline N)_z J-\overline NJ_z}{J^2}\,.
\]
Therefore, by \eqref{eq-N1} and the condition $N(\a)=0$, we obtain
\begin{equation} \label{eq-N2}
\left(\frac{N}{J}\right)_z(\a)=1\quad \text{and}\quad \left(\frac{\overline N}{J}\right)_z(\a)=0\,.
\end{equation}

Since
\[
\left(\frac{N}{J}\right)_{zz}=-\frac{[N_zJ_z+NJ_{zz}]J-2NJ^2_z}{J^3}\,,
\]
we have, by \eqref{eq-N1},
\begin{equation} \label{eq-N3}
\left(\frac{N}{J}\right)_{zz}(\alpha)=-\frac{N_z(\alpha)J_z(\alpha)}{J^2(\alpha)}=-\frac{J_z(\alpha)}{J(\alpha)}\,.\end{equation}

The first two derivatives of the function $A$ are given, respectively, by
\[
A_z=-\frac{\overline{h''}}{2}\left(\frac{\overline N}{J}\right)_z f+ \left(\overline{h'}-\frac{\overline{h''}}{2}\frac{\overline N}{J} \right)h'+\frac{\overline{g''}}{2}\left(\frac{\overline N}{J}\right)_z \overline f- \left(\overline{g'}-\frac{\overline{g''}}{2}\frac{\overline N}{J} \right)g'
\]
and
\begin{eqnarray*}
A_{zz}&=& \left[-\frac{\overline{h''}}{2} \left(\frac{\overline N}{J}\right)_{zz}\right]f-\left[\overline{h''}\left(\frac{\overline N}{J}\right)_{z}\right]h'+   \left[\overline{h'}-\frac{\overline{h''}}{2}\frac{\overline N}{J} \right]h''\\
&+&
\left[\frac{\overline{g''}}{2} \left(\frac{\overline N}{J}\right)_{zz}\right]\overline{f}+\left[\overline{g''}\left(\frac{\overline N}{J}\right)_{z}\right]g'-   \left[\overline{g'}-\frac{\overline{g''}}{2}\frac{\overline N}{J} \right]g''\,.
\end{eqnarray*}

Therefore, using $f(\a)=N(\a)=0$, \eqref{eq-N1}, and \eqref{eq-N2}, we obtain
\begin{equation}\label{eq-A1}
A_z(\a)=J(\a)\quad\text{and}\quad A_{zz}(\a)=J_z(\a)\,.
\end{equation}

Moreover, these conditions can be used to get the value of $A_{zzz}(\alpha)$, which equals
\begin{eqnarray}\label{eq-A2}
\nonumber A_{zzz}(\alpha)&=&-\frac{3}{2}\left(\overline{h''(\a)}h'(\a)-\overline{g''(\a)}g'(\a)\right)\, \left(\frac{\overline N}{J}\right)_{zz}(\a)\\
\nonumber &+& \left(\overline{h'(\a)}h'''(\a)-\overline{g'(\a)}g'''(\a)\right)\\
&=& -\frac{3}{2} \overline{J_z(\a)} \left(\frac{\overline N}{J}\right)_{zz}(\a)+ \left(\overline{h'(\a)}h'''(\a)-\overline{g'(\a)}g'''(\a)\right)\,.
\end{eqnarray}

Regarding the function $B$, we have

\begin{eqnarray*}
B_{z}&=& \left[h''-\frac{h'''}{2} \left(\frac{N}{J}\right)-\frac{h''}{2} \left(\frac{N}{J}\right)_z\right] \left[\overline{h'}-\frac{\overline{h''}}{2} \left( \frac{\overline N}{J}\right) \right]\\
&+&  \left[h'-\frac{h''}{2} \left(\frac{N}{J}\right)\right] \left[-\frac{\overline{h''}}{2} \left( \frac{\overline N}{J}\right)_z \right]\\
&-&
 \left[g''-\frac{g'''}{2} \left(\frac{N}{J}\right)-\frac{g''}{2} \left(\frac{N}{J}\right)_z\right] \left[\overline{g'}-\frac{\overline{g''}}{2} \left( \frac{\overline N}{J}\right) \right]\\
 &-&  \left[g'-\frac{g''}{2} \left(\frac{N}{J}\right)\right] \left[-\frac{\overline{g''}}{2} \left( \frac{\overline N}{J}\right)_z \right]\,.
\end{eqnarray*}

Hence, using again $f(\a)=N(\a)=0$, \eqref{eq-N1}, and \eqref{eq-N2}, we get
\begin{equation}\label{eq-B1}
B_z(\a)=\frac{h''(\alpha) \overline{h'(\alpha)}-g''(\alpha) \overline{g'(\alpha)}}{2} = \frac{J_z(\alpha)}{2}
\end{equation}
and
\begin{align}\label{eq-B2}
\nonumber& B_{zz}(\alpha)= -\frac{\overline{h'(\a)}h''(\a)}{2} \left(\frac{N}{J}\right)_{zz}(\a)-  \frac{ h'(\a)\overline{h''(\a)}}{2} \left(\frac{\overline N}{J}\right)_{zz}(\a)\\
\nonumber&+ \frac{\overline{g'(\a)}g''(\a)}{2} \left(\frac{N}{J}\right)_{zz}(\a)+  \frac{ g'(\a)\overline{g''(\a)}}{2} \left(\frac{\overline N}{J}\right)_{zz}(\a)\\
\quad & = -\frac{J_z(\a)}{2} \left(\frac{N}{J}\right)_{zz}(\a) -\frac{\overline{J_z(\a)}}{2} \left(\frac{\overline N}{J}\right)_{zz}(\a)\,.
\end{align}

With this information at hand, now it is easy to prove that $F_z(\alpha)=F_{zz}(\alpha)=0$. More specifically, since
\[
F_z=1-\frac{A_zB-AB_z}{B^2}\,,
\]
we have, using $A(\a)=0$, $B(\a)=J(\a)$, and \eqref{eq-A1}, the desired first identity $F_z(\a)=0$. The second derivative with respect to $z$ of $F$ equals
\[
F_{zz}=\frac{2B_z(A_zB-AB_z)-B(A_{zz}B-AB_{zz})}{B^3}\,,
\]
which gives, using also \eqref{eq-B1}, $F_{zz}(\a)=0$. 

\par\smallskip
The fact that $A(\a)=0$ gives 
\begin{eqnarray*}
F_{zzz}(\a)&=&\frac{2A_z(\a)B^2(\a)B_{zz}(\a)-A_{zzz}(\a)B^3(\a)+A_z(\a)B^2(\a)B_{zz}(\a)}{B^4(\a)}\\
&+&\frac{-6A_z(\a)B(\a)B^2_z(\a)+3A_{zz}(\a)B^2(\a)B_z(\a)}{B^4(\a)}\\
&=&\frac{3A_z(\a)B(\a)B_{zz}(\a)-A_{zzz}(\a)B^2(\a)}{B^3(\a)}\\
&+&\frac{3A_{zz}(\a)B(\a)B_z(\a)-6A_z(\a) B^2_z(\a)}{B^3(\a)}\,.\\
\end{eqnarray*}

Using that $B(\a)=J(\a)$, \eqref{eq-A1}, and \eqref{eq-B1}, we get

\begin{eqnarray*}
F_{zzz}(\a)&=&\frac{3J^2(\a)B_{zz}(\a)-A_{zzz}(\a)J^2(\a)}{J^3(\a)}\\
&+&\frac{-\frac{3}{2}J(\a) J_z^2(\a)+\frac{3}{2} J(\a) J^2_z(\a)}{J^3(\a)}\\
&=& \frac{3B_{zz}(\a)-A_{zzz}(\a)}{J(\a)}\,.
\end{eqnarray*}

Finally, by \eqref{eq-A2} and \eqref{eq-B2}, we obtain (using also \eqref{eq-N3})

\begin{align}\label{eq-F}
\nonumber & F_{zzz}(\a)=\frac{3}{J(\alpha)} \left(-\frac{J_z(\a)}{2} \left(\frac{N}{J}\right)_{zz}(\a) -\frac{\overline{J_z(\a)}}{2} \left(\frac{\overline N}{J}\right)_{zz}(\a)\right)\\
 \nonumber &- \frac{1}{J(\a)} \left( -\frac{3}{2} \overline{J_z(\a)} \left(\frac{\overline N}{J}\right)_{zz}(\a)+ \left(\overline{h'(\a)}h'''(\a)-\overline{g'(\a)}g'''(\a)\right)\right)\\
 \nonumber & = -\frac{3}{2} \frac{J_z(\a)}{J(\alpha)} \left(\frac{N}{J}\right)_{zz}(\a)-\frac{\overline{h'(\a)}h'''(\a)-\overline{g'(\a)}g'''(\a)}{J(\alpha)}\\
 & =  \frac{3}{2} \frac{J^2_z(\a)}{J^2(\alpha)} -\frac{\overline{h'(\a)}h'''(\a)-\overline{g'(\a)}g'''(\a)}{J(\alpha)}\,.
\end{align}

To finish the proof of the theorem, we are to show that the expression in \eqref{eq-F} equals $-S_H(f)(\a)$, where $S_H(f)$ is the Schwarzian derivative of $f$ given by \eqref{eq-Schwarzian}. To do so, recall that $g'=\omega h'$, where $\omega$ is de dilatation of the harmonic function $f=h+\overline g$. Therefore,
\begin{equation}\label{eq-dilat}
g''=\omega' h'+\omega h'' \quad \text{and}\quad g'''=\omega'' h' +2\omega'h''+\omega h'''\,.
\end{equation}

The first term in \eqref{eq-F}, then, equals
\begin{eqnarray}\label{eq-1}
\nonumber \frac{3}{2}\left(\frac{J_z(\a)}{J(\a)}\right)^2&=&\frac{3}{2} \left(\frac{\overline{h'(\a)}h''(\a)-\overline{g'(\a)}g''(\a)}{\overline{h'(\a)}h'(\a)-\overline{g'(\a)}g'(\a)}\right)^2\\
\nonumber &=& \frac{3}{2} \left(\frac{h''(\a)-\overline{\omega(\a)}g''(\a)}{h'(\a)-\overline{\omega(\a)}g'(\a)}\right)^2\\
\nonumber &=&  \frac{3}{2} \left(\frac{h''(\a)(1-|\omega(\a)|^2)-\overline{\omega(\a)}\omega'(\a)h'(\a)}{h'(\a)(1-|\omega(\a)|^2)}\right)^2\\
&=&\frac{3}{2} \left(\frac{h''(\a)}{h'(\a)}-\frac{\overline{\omega(\a)}\omega'(\a)}{1-|\omega(\a)|^2}\right)^2\,.
\end{eqnarray}

On the other hand, the second term in \eqref{eq-F} can be written, using \eqref{eq-dilat} as

\begin{align}\label{eq-2}
& \nonumber \frac{\overline{h'(\a)}h'''(\a)-\overline{g'(\a)}g'''(\a)}{J(\alpha)}= \frac{\overline{h'(\a)}h'''(\a)-\overline{g'(\a)}g'''(\a)}{\overline{h'(\a)}h'(\a)-\overline{g'(\a)}g'(\a)}\\
& \nonumber =\frac{h'''(\a)-\overline{\omega(\a)}g'''(\a)}{h'(\a)-\overline{\omega(\a)}g'(\a)}\\ 
&\nonumber = \frac{h'''(\a)(1-|\omega(\a)|^2)-2\overline{\omega(\a)}\omega'(\a)h''(\a)-\overline{\omega(\a)}\omega''(\a)h'(\a)}{h'(\a)(1-|\omega(\a)|^2)} 
\\ & =  \frac{h'''(\a)}{h'(\a)} -2 \frac{\overline{\omega(\a)}}{1-|\omega(\a)|^2} \frac{h''(\a)}{h'(\a)} \omega'(\a) -\frac{\overline{\omega(\a)}}{1-|\omega(\a)|^2} \omega''(\a)\,.
\end{align}

By plugging \eqref{eq-1} and \eqref{eq-2} into \eqref{eq-F}, we end the proof of the theorem.
\end{proof}

\section{Final remarks}

As pointed out before, in this paper we have not considered the analysis of the conditions of convergence of the harmonic numerical methods presented. For such analysis, we refer the reader to the paper \cite{SZ} for the harmonic Newton method. The analysis of the convergence conditions (if any)  in relation to the iteration given by the generating function \eqref{eq-fcn} may perhaps be an object of study in a future work. 

\par\smallskip
Of particular interest might also be to analyze the possible analogues of the geometric properties of the classical Halley method, already studied by different authors (see, for instance, \cite{ST} and the references therein. The articles \cite{S} and \cite{YB} must be mentioned as well) to those cases when the functions considered are harmonic.

\section{Acknowledgements}

The author would like to thank the referee for his/her careful reading of the paper.

\end{document}